\documentclass[11pt]{amsart}
\usepackage{latexsym,amssymb,amsmath}
\usepackage[colorlinks=true, linkcolor=blue, anchorcolor=black, citecolor=blue, filecolor=blue, menucolor= blue, urlcolor=black]{hyperref}
\usepackage{graphicx}
\usepackage{float}

\newtheorem{theorem}{Theorem}[section]
\newtheorem{lemma}[theorem]{Lemma}
\newtheorem{proposition}[theorem]{Proposition}

\newtheorem{conjecture}[theorem]{Conjecture}

\newtheorem{notation}[theorem]{Notation}

\theoremstyle{definition}
\newtheorem{definition}[theorem]{Definition}

\newtheorem{remark}[theorem]{Remark}
\newtheorem{example}[theorem]{Example}

\newtheorem*{ConventionI}{Convention}

\newtheorem{theoremx}{Theorem}

\newcommand{\NN}{\mathbb{N}}

\newcommand{\ZZ}{\mathbb{Z}}
\newcommand{\KK}{\mathbb{K}}
\newcommand{\QQ}{\mathbb{Q}}
\newcommand{\FF}{\mathbb{F}}

\newcommand{\m}{\mathfrak{m}}

\begin{document}


\title{$F$-THRESHOLDS AND TEST IDEALS OF  THOM-SEBASTIANI TYPE POLYNOMIALS}

\author[M. Gonz\'alez Villa]{Manuel Gonz\'alez Villa$^{1}$}
\address{
Centro de inestigaci\'on en 
Matem\'aticas\\
Apartado Postal
402 \\
36000 Guanajuato, GTO.
}
\thanks{$^{1}$ The first author was partially supported by 
Spanish national grant
MTM2016-76868-C2-1-P}
\email{manuel.gonzalez@cimat.mx}

\author[D. Jaramillo-Velez]{Delio Jaramillo-Velez$^{2}$}
\address{
Departamento de
Matem\'aticas\\
Centro de Investigaci\'on y de Estudios
Avanzados del
IPN\\
Apartado Postal
14--740 \\
07000 Mexico City, CDMX.
}
\thanks{ $^{2}$ The second author was partially supported by  CONACyT Fellowship 862006}.
\email{djaramillo@math.cinvestav.mx}

\author[L. N\'u\~nez-Betancourt]{Luis N\'u\~nez-Betancourt$^{3}$}
\address{
Centro de inestigaci\'on en 
Matem\'aticas\\
Apartado Postal
402 \\
36000 Guanajuato, GTO.
}
\thanks{$^{3}$ The third author was partially supported by  CONACyT Grant 284598 and C\'atedras Marcos Moshinsky.}
\email{luisnub@cimat.mx}


\maketitle

\begin{abstract}
We provide a formula for  $F$-thresholds of a Thom-Sebastiani type polynomial over a perfect field of prime characteristic. This result extends the formula for the $F$-pure threshold of a diagonal hypersurface. We also compute the first test ideal of  Thom-Sebastiani type polynomials.  Finally, we apply our result to find hypersurfaces where the log canonical thresholds equals the $F$-pure thresholds for infinitely many prime numbers.
\end{abstract}


\section{Introduction}\label{intro-section}
The $F$-threshold of a polynomial $f$, with respect to an ideal $I$, $c^I(f)$, is a
numerical invariant given by  the asymptotic Frobenius order of $f$ in $I$ \cite{MTW,HMTW,DSNBP}. In particular, if $\m$ is a maximal ideal, $c^\m(f)$ equals to the $F$-pure threshold at $\m$.
This invariant  measures the severity of the singularities of $f$ at the point corresponding  to $\m$ \cite{TW}.   This $F$-threshold is the analogue in prime characteristic of the log-canonical threshold \cite{HY,TW,MTW}.
This number is usually very difficult to compute.  However, there are  formulas for diagonal \cite{HerDiag}, 
binomial \cite{HerBinom},  Calabi-Yau \cite{16}, Elliptic Curves \cite{16,Pagi}, and quasi-homogeneous one dimensional hypersurfaces \cite{HNBWZ}.

An important aspect of a numerical invariant is its computability in concrete examples. Finding formulas  for  Thom-Sebstiani type polynomials is a classic strategy to enlarge the class of concrete hypersurfaces for which a  invariant can be computed. 
This type of polynomial can be written as the sum of two polynomials in a different set of variables. Examples of this type of polynomials include, for instance, diagonal hypersurfaces, certain binomials, and suspensions by $n\geq 2$ points or ramified cyclic covers of hypersurfaces.
In singularity theory formulas for  Thom-Sebastiani type polynomials have been found  for, among others,  the spectrum \cite{SS}, the monodromy,  the Milnor number \cite{ST},  the log canonical threshold \cite{Kollar},  motivic  Igusa zeta functions \cite{DL}, multiplier ideals, and  jumping numbers \cite{MSS}. It is natural to look for such formulas for objects in prime characteristic.

In this paper we compute  $F$-thresholds and the first test ideal of a Thom-Sebastiani type polynomial. Our methods are inspired in the 
work of  Hern\'andez \cite{HerDiag}  on diagonal hypersurfaces.
 
 In our first main result, we provide a formula for the $F$-threshold  of a Thom-Sebastiani type polynomial with respect to the sum of two ideals.

\begin{theoremx}[see  Theorem \ref{4.2}]\label{A}
Let $\KK$ be a perfect field of prime characteristic $p$.
Let $R_{1}=\KK[x_{1},\dots,x_{n}]$ and $R_{2}=\KK[y_{1},\dots,y_{m}]$ with maximal homogeneous ideals 
$ \mathfrak{m}_{1}=\left(x_{1},\dots,x_{n}\right)$ and $\mathfrak{m}_{2}=\left(y_{1},\dots,y_{m}\right)$ respectively.
Let  $g_1\in \sqrt{I_1}\subseteq \m_1$ and $ g_2\in \sqrt{I_2}\subseteq \m_2$, where $I_1$ and $I_2$ are ideals.
Let $f=g_1+g_2 \in R_1\otimes_\KK R_2$, $a_1={\bf c}^{I_1}(g_1)$, and $a_2={\bf c}^{I_2}(g_2)$. 
If $a_1+a_2\leq 1$, then 
$$
	{\bf c}^{I_1+I_2}(f)=\left\lbrace \begin{array}{ll}
	a_{1}+a_{2},& \text{if}\quad L=\infty,\\
	&\\
	\left\langle a_{1}\right\rangle_{L} + \left\langle a_{2}\right\rangle_{L} +\frac{1}{p^{L}},& \text{if}\quad L<\infty,
	\end{array}\right.
	$$
	where $L=\sup\left\lbrace N\in \mathbb{N}\;|\; a_{1}^{(e)}+a_{2}^{(e)}\leq p-1\;\text{for all}\; 0\leq e\leq N \right\rbrace.$
\end{theoremx}

Unlike previous work on computation of $F$-thresholds,
we do not assume any shape on the $F$-pure thresholds of either $g_1$ or $g_2$.  Furthermore, we are able to compute $F$-thresholds with respect to ideals other than the maximal ideal. This is useful to compute lower bound of the number of $F$-jumping numbers for a Thom-Sebastiani type polynomial (see Remark \ref{RemNumberFTbound}).

An important conjecture  regarding the $F$-pure threshold states that ${\bf lct}(f)={\bf c}^{\m}(f_p)$
for infinitely many prime numbers $p$, where  $f\in\QQ[x_1,\ldots, x_n]$ and $f_p\in\FF_p[x_1,\ldots,x_n]$ is its model module $p$ (see \cite[Conjecture 3.6]{MTW} and references therein). 
Using Theorem  \ref{A} , we are able to show several families of polynomials where this conjecture holds (see Section \ref{SecLCT}).

We are able to compute the first test ideal of a Thom-Sebastiani type polynomial
at the $F$-pure threshold. We point out that, unlike previous work, we do not assume that the test ideals of $g_1$ or $g_2$ are monomial. In fact, we do not assume any shape of the ideals. Furthermore, we do not make any assumption on the characteristic.

\begin{theoremx}[see Theorem \ref{456}]\label{B}
Let $\KK$ be a perfect field of prime characteristic $p$.
Let $R_{1}=\KK[x_{1},\dots,x_{n}]$ and $R_{2}=\KK[y_{1},\dots,y_{m}]$ with maximal homogeneous ideals 
$ \mathfrak{m}_{1}=\left(x_{1},\dots,x_{n}\right)$ and $\mathfrak{m}_{2}=\left(y_{1},\dots,y_{m}\right)$ respectively.
Let $f=g_1+g_2 \in R=R_1 \otimes_\KK R_2=\KK[x_1, \dots, x_n, y_1, \dots, y_m]$, where  $g_1\in \mathfrak{m}_{1}$ and $g_2\in \mathfrak{m}_{2}$. Let  $a_{1}={\bf c}^{\mathfrak{m}_{1}}(g_{1})=\frac{r_1}{s_1}$,  $a_{2}={\bf c}^{\mathfrak{m}_{2}}(g_{2})=\frac{r_2}{s_2}$, and $\mathfrak{m}=\left(x_{1},\dots,x_{n},y_{1},\dots,y_{m}\right)$, where $r_i,s_i\in\NN$. 
Suppose that  $a_1+a_2\leq 1$, and   set 
$$L=\sup\left\lbrace N\in \mathbb{N}\;|\; a_{1}^{(e)}+a_{2}^{(e)}\leq p-1\;\text{for all}\; 0\leq e\leq N \right\rbrace,$$
and  $d=\sup\left\lbrace e\leq L\quad|\quad a_{1}^{(e)}+a_{2}^{(e)}\leq p-2\right\rbrace.$
Then,
$$
\tau\left( f^{{\bf c}^{\mathfrak{m}}(f)}\right) =\left\lbrace \begin{array}{cl}
(f)&\text{if } {\bf c}^{\mathfrak{m}}(f)=1,  \\
\tau(g^{a_1}_{1})+\tau(g^{a_2}_{2})&\text{if } {\bf c}^{\mathfrak{m}}(f)\not\in p^{-e}\cdot\mathbb{N},\\
\left( g_{1}^{\lceil p^{d} a_{1}\rceil}\right)^{[1/p^{d}]}+\left( g_{2}^{\lceil p^{d} a_{2}\rceil}\right)^{[1/p^{d}]}&\text{if }  {\bf c}^{\mathfrak{m}}(f)\in \ZZ[\frac{1}{p}]\; \& \;{\bf c}^{\mathfrak{m}}(f)\neq 1 .
\end{array}\right.
$$
\end{theoremx}

We point out that Theorem \ref{B} includes several cases that were not computed before,  even for diagonal hypersurfaces (see Example \ref{ExNotHdz}).

\begin{ConventionI}
 In this manuscript, $p$ denotes a prime number and  $\KK$ denotes a perfect field of characteristic $p$.
\end{ConventionI}

\section{Background}

\subsection{Expasions in base $p$}\label{2}

In this section we recall  the notion of  the non-terminating base $p$ expansion of a number in $(0,1]$, and some of its properties. 

\begin{definition}\label{2.1}
Let $\alpha \in \left( 0,1 \right] $, and $p$ be a prime number. The non-terminating base $p$ expansion of $\alpha$ is the  expression $\alpha=\displaystyle\sum_{e\geq 1}\frac{\alpha^{(e)}}{p^{e}}$, with $0\leq\alpha^{(e)}\leq p-1 $, such that for all $ n>0$, there exists $ e\geq n$ with $\alpha^{(e)}\not= 0$. The number $\alpha^{(e)}$ is unique and it is called the $e^{th}$ digit of the non-terminating base $p$ expansion of $\alpha$.
\end{definition}

\begin{example}\label{2.2}
Let $\alpha=\frac{a}{b}$ be a  rational number in $\left( 0,1 \right] $. If $p\equiv 1 \mod b$, then  $p=bw+1$ for some $w\geq 1$. 
The non-terminating base $p$ expansion of $\alpha$ is periodic, and it is given by $\alpha=\displaystyle\sum_{e\geq 1}\frac{aw}{p^{e}}$.
\end{example}

The elementary notion of adding without carrying for integers (base $p$) extends to non-terminating base $p$ expansions.

\begin{definition}\label{2.3}
	Let $\alpha\in \left( 0,1\right] $.
	\begin{enumerate}
		\item For $e\geq1$, the $e^{th}$ truncation of the non-terminating base $p$ expansion of $\alpha$ is defined  by 
		$$
		\left\langle \alpha \right\rangle_{e}=\frac{\alpha^{(1)}}{p}+\dots+\frac{\alpha^{(e)}}{p^{e}}. 
		$$
		 We use the conventions $\left\langle \alpha \right\rangle_{\infty}=\alpha $ and $\left\langle 0 \right\rangle_{e}=0 $.
		\item Let $p<n$, and   $\alpha_{1},\alpha_{2}\in\left(0,1\right)$. We say that 
		$\alpha_{1}$ and $\alpha_{2} $ add without carrying in base $p$ if 
		$$
		\alpha_{1}^{(e)}+\alpha_{2}^{(e)}\leq p-1 \;\text{ for every}\quad e\geq 1.
		$$
	\end{enumerate}
\end{definition}

\begin{remark}\label{addwithout}
The non-terminating base $p$ expansions of $\alpha_1$ and $ \alpha_n$ add without
carrying if and only if the base $p$ expansions of the integers $p^e\left\langle  \alpha_1 \right\rangle_e$ and  $p^e \left\langle \alpha_n \right\rangle_e$ add without carrying for all $e>1$.
\end{remark}

\begin{lemma}[{\cite[Lemma 4.6]{7}}]\label{2.4}
	Let $\alpha\in(0,1]$. Then the following statements hold.
	\begin{enumerate}
		\item If $\alpha\not\in\frac{1}{p^{e}}\cdot\mathbb{N}$, then $\lfloor p^{e}\alpha\rfloor=p^{e}\left\langle \alpha\right\rangle_{e}$.
		\item $\lceil p^{e}\alpha \rceil=p^{e}\left\langle \alpha\right\rangle_{e}+1$.
		\item Let  $\alpha_{1},\alpha_{2} \in\left(0,1\right]$. If $\alpha_{1}$ and $\alpha_{2}$ add without carrying in base $p$ and $\alpha:=\alpha_{1}+\alpha_2\leq1$, then $\alpha^{(e)}=\alpha_{1}^{(e)}+\alpha_{2}^{(e)}$.
	\end{enumerate}
\end{lemma}

To end this section we include some useful results on divisibility for integers. 

\begin{lemma}[{\cite{Luc78,Dic}}]\label{2.5}
	Let $k_{1},k_{2}\in \mathbb{N}$ and set $N=k_{1}+k_{2}$. Then 
	$$
	{N \choose k_{i}}=\frac{N!}{k_{1}!k_{2}!}\not\equiv 0\mod p
	$$
	if and only if $k_{1},k_{2}$ add without carrying in base $p$.
\end{lemma}

\begin{theorem}[Dirichlet]\label{2.6.1}
Given $n$ and $m$ non-zero natural numbers with ${\rm gcd}(n,m)=1$, there exists infinitely many prime numbers $p$ such that $p\equiv n\mod m$.
\end{theorem}

\begin{theorem}[Dirichlet]\label{2.6}
	For any collection $\alpha_{1},\dots,\alpha_{n}$ of rational numbers, there exists infinitely many prime numbers  $p$ such that $(p-1)\cdot\alpha_{i}\in\mathbb{N}$ for $1\leq i\leq n$.
\end{theorem}

\subsection{Test ideals and F-thresholds}\label{prelim-section}

\begin{notation}
Throughout this subsection, $R$ denotes the polynomial ring $\mathbb{K}[x_{1},\dots,x_{n}]$ over  $\mathbb{K}$, and $\mathfrak{m}$ denotes its homogeneous maximal ideal $(x_{1},\dots,x_{n})$. 
\end{notation}

The Frobenius map $F:R\rightarrow R$ is the $p^{th}$ power morphism 
given by $r\mapsto r^{p}$.       
For every ideal $I\subseteq R$, the $e$-th Frobenius power of $I$, denoted by  $I^{\left[ p^{e}\right] }$,  is the ideal generated by the set $\left\lbrace g^{p^{e}} \, | \, g\in I \right\rbrace$. 
Since $R$ is a regular ring, then 
the Frobenius map $F:R\rightarrow R$ is flat. Therefore, we have that
$(I^{[p^{e}]}:_{R}J^{[p^{e}]})=(F^{e}(I):_{R}F^{e}(J))=F^{e}((I:_{R} J))=(I:_{R} J)^{[p^{e}]}.$
\begin{remark}\label{ret}
The ring $R$ is finitely generated and free over $R^{p^{e}}=\mathbb{K}[ x_{1}^{p^{e}},\dots,x_{n}^{p^{e}}]$, with basis the set of monomials $\left\lbrace \mu \, | \, \mu\not\in \mathfrak{m}^{[p^{e}]}\right\rbrace $.
\end{remark}

\begin{definition} \label{fraexpo}
 For an ideal $I
 \subseteq R$ and a positive integer $e$, let $I^{\left[ 1/p^{e}\right] }$ denote the unique smallest ideal $J$ of $R$ with respect to the inclusion such that  
$I \subseteq J^{[p^e]}.$
 \end{definition}
 
 \begin{proposition}[{\cite[Proposition 2.5]{BMS-TAMS}}]\label{compute Frobeniusroot}
 Let $e_{1},\dots,e_{s}$ be a basis of $R$ over $R^{p^{e}}$. Let $f_1, \dots, f_r$ be generators of the ideal $I$ of $R$. For each $i \in \{1,  \dots, s\}$  let $f_{i}=\sum^{s}_{j=1}a_{i,j}^{p^{e}}e_{j}$ with $a_{i,j} \in R$. Then,
$I^{\left[ 1/p^{e}\right] }= (a_{i,j} \, | \,  1 \leq i \leq r, 1 \leq j \leq s).$
\end{proposition}

\begin{definition}[{\cite{HY,BMS}}]\label{testid}
 Let $I\subseteq R$ be an ideal, $c$ a positive real number, and $e$ a positive integer. Then, the test ideal of $I$ with exponent $c$ is 
$\tau\left( I^{c}\right)=\bigcup_{e\geq 1} \left( I^{\lceil cp^{e}\rceil}\right)^{\left[ 1/p^{e}\right] }. $
\end{definition}

\begin{remark}\label{2.7}
Let $I$ be an ideal of $R$. Then, $f^{p^{e}}\in I^{\left[ p^{e}\right] }$ if and only if $f\in I$. This follows because $R$ is finitely generated and free over $R^{p^{e}}$. 
\end{remark}

If $\alpha$ is a rational number with a power of $p$ in the denominators, the test ideal can be obtained from the following formula.

\begin{proposition}[{\cite[Lemma 2.1]{BMS}}]\label{PropTestIdealPpower}
Let $r,e\in\NN$. Then, $\tau(f^{r/p^e})=  \left( f^{r}\right)^{\left[ 1/p^{e}\right] }.$
\end{proposition}

\begin{definition}\label{fpt}
	The $F$-threshold of a non-zero proper ideal $\mathfrak{a} \subseteq\sqrt{J}\subseteq \mathfrak{m}$  with respect to an ideal $J$, denoted by ${\bf c}^{J}(\mathfrak{a})$, is defined as 
$$
	\displaystyle\lim_{e\rightarrow \infty}\frac{\nu_{\mathfrak{a}}^{J}(p^{e})}{p^{e}},
$$
	where $\nu_{\mathfrak{a}}^{J}(p^{e})=	\max\left\lbrace l\in\mathbb{N}|\mathfrak{a}^{l}\not\subseteq J^{[p^{e}]}\right\rbrace. $
\end{definition}

It is known that the previous limit exits \cite[Lemma 1.1]{MTW}, and  ${\bf c}^{J}(\mathfrak{a})$ is a rational number \cite[Theorem 3.1]{BMS}. If $\mathfrak{a}=\left(f\right)$, we simply write $\nu_{f}^{J}(p^{e})$, ${\bf c}^{J}(f)$ and $\tau\left( f^{c}\right)$.

\begin{proposition}[{\cite[Proposition 1.9]{MTW}}]\label{tra}
Let $J\subseteq \mathfrak{m}$ be an ideal whose radical contains $f\not= 0$. For every non-negative integer $e$,  we have 
$$
\left\langle {\bf c}^{J}(f)\right\rangle _{e}=\frac{\nu_{f}^{J}(p^{e})}{p^{e}}
$$
\end{proposition}

\section{F-thresholds of a Thom-Sebastiani type polynomial}\label{4}
In this section we focus in stuying the $F$-threshold of  a Thom-Sebastiani type polynomial with respect to the sum of two ideals in different variables. In particular, we prove Theorem \ref{A}.  We start fixing notation for this section.

\begin{notation}\label{NotationFT}
 Let $\{x_1, \dots, x_n\}$ and $\{ y_1, \dots, y_m\}$ be two disjoint sets of variables. Consider the rings of polynomials $R_{1}=\KK[x_{1},\dots,x_{n}]$ and $R_{2}=\KK[y_{1},\dots,y_{m}]$. Let $\mathfrak{m}_{1}=\left(x_{1},\dots,x_{n}\right)$ and $\mathfrak{m}_{2}=\left( y_{1},\dots,y_{m} \right)$  denote their maximal homogeneous ideals.
Let $I_1\subseteq \m_1$ and $I_2\subseteq \m_2$ be two ideals. 
Let $a_{1}={\bf c}^{I_1}(g_1)$ and  $a_{2}={\bf c}^{I_2}(g_2)$.  Let $f=g_1+g_2 \in R=R_1 \otimes_\KK R_2=\KK[x_1, \dots, x_n, y_1, \dots, y_m]$, where  $g_1\in R_{1}$ and $g_2\in R_{2}$.
\end{notation}

We start with a lemma that may be know to the experts. We add it for the sake of completeness.

\begin{lemma}\label{LemmaPorwerAndTerms}
We consider Notation \ref{NotationFT}.
Let $\theta, e\in\NN$. Then,  $f^\theta\in I^{[p^e]}_1R +  I^{[p^e]}_2 R$
 if and only if
$\binom{\theta}{j}=0, $ $g_1^j\in I^{[p^e]}_1 $, or $g_2^{\theta-j}\in I^{[p^e]}_2$ for every 
$0\leq j\leq \theta$.
\end{lemma}
\begin{proof}
We first assume that $\binom{\theta}{j}=0, $ $g_1^j\in I^{[p^e]}_1 $, or $g_2^{\theta-j}\in I^{[p^e]}_2$ for every 
$0\leq j\leq \theta$. Then, $f^\theta=\sum^\theta_{j=0}\binom{\theta}{j} g_1^j g_2^{\theta-j}\in I^{[p^e]}_1R +  I^{[p^e]}_2 R$.

We now assume that $f^\theta\in I^{[p^e]}_1R +  I^{[p^e]}_2 R$.
We proceed by contradiction.
We assume that the set
$$\mathcal{A}:=\left\{ j \in\NN \; |\; \binom{\theta}{j}\neq 0, g_1^j\not\in I^{[p^e]}_1 , g_2^{\theta-j}\not\in I^{[p^e]}_2, \hbox{ and } j \leq \theta \right\}$$
is not empty. Set $t:= \min \mathcal{A}$.

By our choice of $t$, we have that $\sum_{j=0}^{t-1} \binom{\theta}{j} g^j_1 g^{\theta-j}_2 \in I^{[p^e]}_1R +  I^{[p^e]}_2 R$, and therefore 
$$h= f - \sum_{j=0}^{t-1} \binom{\theta}{j} g^j_1 g^{\theta-j}_2 = \sum^\theta_{j=t}\binom{\theta}{j} g^j_1 g^{\theta-j}_2 \in I^{[p^e]}_1R +  I^{[p^e]}_2 R.$$
Let $\beta=\nu^{I_1}_{g_1}(p^e)$. 
We have that $t\leq \beta$ and
$$ g^{\beta-t}_{1}h=\sum^\theta_{j=t}\binom{\theta}{j} g^{\beta+j-t}_1 g^{\theta-j}_2
\in  I^{[p^e]}_1R +  I^{[p^e]}_2 R.$$
Moreover, $\binom{\theta}{j} g^{\beta+j-t}_1 g^{\theta-j}_2\in    I^{[p^e]}_1R +  I^{[p^e]}_2 R$ for every $j>t$.
Then, $\binom{\theta}{t} g^{\beta}_1 g^{\theta-t}_2 \in  I^{[p^e]}_1R +  I^{[p^e]}_2 R$.
Since $\binom{\theta}{t} \neq 0$, we deduce that $ g^{\beta}_1 g^{\theta-t}_2 \in  I^{[p^e]}_1R +  I^{[p^e]}_2 R.$

Since
$g_1^\beta\not\in I^{[p^e]}_1$  and $ g_2^{\theta-t}\not\in I^{[p^e]}_2$, we have that
$g^{\beta}_1\neq 0$  in $R_1/I^{[p^e]}_1$ and
$ g^{\theta-t}_2\neq 0 $ in $R_2/I^{[p^e]}_2$. Then,
$0\neq g^{\beta}_1\otimes  g^{\theta-t}_2$ in 
$R_1/I^{[p^e]}_1\otimes_\KK R_2/I^{[p^e]}_2=R/( I^{[p^e]}_1R +  I^{[p^e]}_2 R).$
Hence,
 $ g^{\beta}_1 g^{\theta-t}_2 \not \in  I^{[p^e]}_1R +  I^{[p^e]}_2 R,$ a contradiction.
\end{proof}

We start proving Theorem \ref{A} by finding a lower bound for $F$-thresholds of a  Thom-Sebastiani type polynomial with respect to an ideal that has also  Thom-Sebastiani type shape.

\begin{lemma}\label{4.1}
We consider Notation \ref{NotationFT}.
Then,
    $$
	{\bf c}^{I_1+I_2}(g_1+g_2)\geq \left\langle a_{1}\right\rangle_{L} + \left\langle a_{2}\right\rangle_{L} +\frac{1}{p^{L}},
	$$
	where $L=\sup\left\lbrace N\in \mathbb{N}\quad|\quad a_{1}^{(e)}+a_{2}^{(e)}\leq p-1\quad\text{for all}\quad 0\leq e\leq N \right\rbrace.$
\end{lemma}
\begin{proof}
We note that $a_{1}^{(L+1)}+a_{2}^{(L+1)}\geq p$, by our definition  of $L$. There exists $\alpha_{1},\alpha_{2}\in \mathbb{N}$ such that $\alpha_{1}+\alpha_{2}=p-1$ with $0\leq \alpha_{1}<a_{1}^{(L+1)}$ and $0\leq \alpha_{2}\leq a_{2}^{(L+1)}$. For any integer 
$e\geq L+2$, we  set 
\begin{align*}
\eta(e)=(\eta_{1}(e),\eta_{2}(e))&:=(\left\langle a_{1}\right\rangle_{L},\left\langle a_{2} \right\rangle_{L})+\left( \frac{\alpha_{1}}{p^{L+1}}+\frac{p-1}{p^{L+2}}+\dots+\frac{p-1}{p^{e}},\frac{\alpha_{2}}{p^{L+1}} \right),\\
&=(\left\langle a_{1}\right\rangle_{L},\left\langle a_{2} \right\rangle_{L})+\left(\frac{1}{p^{L+1}}\cdot\left(\alpha_{1}+\frac{p^{e-(L+1)}-1}{p^{e-(L+1)}}\right),\frac{\alpha_{2}}{p^{L+1}} \right).
\end{align*}
Thus, 
$\eta_{2}(e)\leq \left\langle a_{2} \right\rangle _{L+1}$ and $\eta_{1}(e)<\left\langle a_{1} \right\rangle _{L+1}$. The base $p$ expansions of the integers $p^{e}\eta_{1}(e)$ and $p^{e}\eta_{2}(e)$ are 
$$
p^{e}\eta_{1}(e)=p^{e-(L+1)}\alpha_{1}+p^{e-(L+2)}(p-1)+\dots+p-1,
\, \hbox{ and } \, 
p^{e}\eta_{2}(e)=p^{e-(L+1)}\alpha_{2}.
$$
Since $\alpha_{1}+\alpha_{2}=p-1$,  the integers $p^{e}\eta_{1}(e)$ and $p^{e}\eta_{2}(e)$ add without carrying in base $p$.  
Thus
\begin{equation}\label{13}
{p^{e}(\eta_{1}(e)+\eta_{2}(e))\choose
p^{e}\eta_{1}(e)}
\not\equiv 0\mod p
\end{equation}
by  Lemma \ref{2.5}.

Since $\eta_{1}(e)\leq \left\langle a_{1}\right\rangle_{e} $ and $\eta_{2}(e)\leq \left\langle a_{2}\right\rangle_{e} $ for every $e\geq L+2$, we have  that
$p^{e}\eta_{1}(e)\leq p^{e}\left\langle a_{1}\right\rangle_{e}=\nu_{g_1}^{I_1}(p^{e})$ and 
$p^{e}\eta_{2}(e)\leq p^{e}\left\langle a_{2}\right\rangle_{e}=\nu_{g_2}^{I_2}(p^{e})$.
Therefore,  $g_1^{p^{e}\eta_{1}(e)}\not\in I_1^{[p^{e}]}$, $g_2^{p^{e}\eta_{2}(e)}\not\in I_2^{[p^{e}]}$.
Thus,
$$
f^{p^{e}(\eta_{1}(e)+\eta_{2}(e))}\not\in(I_1+I_2)^{[p^{e}]}
$$
by  Lemma \ref{LemmaPorwerAndTerms}.
It follows that
$$p^{e}\left\langle{\bf c}^{I_1+I_2}(f) \right\rangle_{e} =\nu_{f}^{I_1+I_2}(p^{e}) \geq p^{e}(\eta_{1}(e)+\eta_{2}(e)).$$ Finally, we have
\begin{align*}\left\langle{\bf c}^{I_1+I_2}(f) \right\rangle_{e} \geq \eta_{1}(e)+\eta_{2}(e)
&= \left\langle a_{1}\right\rangle_{L}+\left\langle a_{2}\right\rangle_{L}+\frac{1}{p^{L+1}}\left(\alpha_{1}+\alpha_{2}  \right) +\frac{p-1}{p^{L+2}}+\dots+\frac{p-1}{p^{e}}\\
&=\left\langle a_{1}\right\rangle_{L}+\left\langle a_{2}\right\rangle_{L}+\frac{p-1}{p^{L+1}}+\frac{p-1}{p^{L+2}}+\dots+\frac{p-1}{p^{e}}.\\
\end{align*}
Taking the limit when $e\rightarrow \infty$, we have that
\begin{equation*}
{\bf c}^{I_1+I_2}(f) \geq  \left\langle a_{1}\right\rangle_{L}+\left\langle a_{2}\right\rangle_{L}+\frac{p-1}{p^{L+1}}\left(1+\frac{1}{p}+\frac{1}{p^{2}}+\dots  \right)
=\left\langle a_{1}\right\rangle_{L}+\left\langle a_{2}\right\rangle_{L}+\frac{1}{p^{L}}.\\
\end{equation*}
\end{proof}

\begin{theorem}\label{4.2}
Consider Notation \ref{NotationFT}.
 If  $a_{1}+a_{2}\leq 1$, then 
$$
	{\bf c}^{I_1+I_2}(f)=\left\lbrace \begin{array}{ll}
	a_{1}+a_{2},& \text{if}\quad L=\infty,\\
	&\\
	\left\langle a_{1}\right\rangle_{L} + \left\langle a_{2}\right\rangle_{L} +\frac{1}{p^{L}},& \text{if}\quad L<\infty,
	\end{array}\right.
	$$
	where $L=\sup\left\lbrace N\in \mathbb{N}\quad|\quad a_{1}^{(e)}+a_{2}^{(e)}\leq p-1\quad\text{for all}\quad 0\leq e\leq N \right\rbrace.$

\end{theorem}
\begin{proof}
We split  the proof into two cases.

\noindent \underline{Case 1:  $L<\infty$:}
By Lemma \ref{4.1} we have that 
\begin{equation}\label{11}
{\bf c}^{I_1+I_2}(f)\geq \left\langle a_{1}\right\rangle_{L} + \left\langle a_{2}\right\rangle_{L} +\frac{1}{p^{L}}
\end{equation}

We proceed by contradiction, and assume that the inequality in Equation \ref{11} is strict. Then, 
$$
\frac{\nu_{f}^{I_1+I_2}(p^{e})}{p^{e}}>\left\langle a_{1}\right\rangle_{L} + \left\langle a_{2}\right\rangle_{L} +\frac{1}{p^{L}} \, \hbox{ for } \, e \gg 0.$$
 Multiplying by $p^{e}$, we obtain that
$$
\nu_{f}^{I_1+I_2}(p^{e})>p^{e}(\left\langle a_{1}\right\rangle_{L} + \left\langle a_{2}\right\rangle_{L})+p^{e-L}=p^{e-L}\left(p^{L}(\left\langle a_{1}\right\rangle_{L} + \left\langle a_{2}\right\rangle_{L})+1 \right) \, \hbox{ for } \, e \geq L.\\
 $$
Therefore, 
\begin{equation*}
f^{p^{e-L}\left(p^{L}(\left\langle a_{1}\right\rangle_{L} + \left\langle a_{2}\right\rangle_{L})+1 \right)}=\left( f^{p^{L}(\left\langle a_{1}\right\rangle_{L} + \left\langle a_{2}\right\rangle_{L})+1}\right)^{p^{e-L}}\not\in(I_1+I_2)^{[p^{e}]}=\left( (I_1+I_2)^{[p^{L}]}\right) ^{[p^{e-L}]}.
\end{equation*}
By Remark \ref{2.7}, it follows that 
$f^{p^{L}(\left\langle a_{1}\right\rangle_{L} + \left\langle a_{2}\right\rangle_{L})+1}\not\in (I_1+I_2)^{[p^{L}]}.$
This implies that there exists ${\bf k}=(k_{1},k_{2})\in\mathbb{N}^{2}$ such that $k_{1}+k_{2}=p^{L}(\left\langle a_{1}\right\rangle_{L} + \left\langle a_{2}\right\rangle_{L})+1 $ and $g_1^{k_{1}}g_2^{k_{2}}\not\in(I_1+I_2)^{[p^{L}]}$. In particular, we have 
$g_1^{k_{1}}\not\in (I_1)^{[p^{L}]}, \, \hbox{ and } \,  g_2^{k_{2}}\not\in (I_2)^{[p^{L}]}.$
By Proposition \ref{tra}, we conclude that
$$
k_{1}\leq \nu_{g_1}^{I_1}(p^{L})=p^{L}\left\langle a_{1}\right\rangle_{L}, \, \hbox{ and } \, k_{2}\leq \nu_{g_2}^{I_2}(p^{L})=p^{L}\left\langle a_{2}\right\rangle_{L}.
$$
Since $k_{1}+k_{2}=p^{L}(\left\langle a_{1}\right\rangle_{L} + \left\langle a_{2}\right\rangle_{L})+1 $, we get 
$$
\left\langle a_{1}\right\rangle_{L} + \left\langle a_{2}\right\rangle_{L}+\frac{1}{p^{L}}=\frac{1}{p^{L}}(k_{1}+k_{2})\leq \left\langle a_{1}\right\rangle_{L} + \left\langle a_{2}\right\rangle_{L},
$$
which is a contradiction. Therefore, it holds  
$$
{\bf c}^{I_1+I_2}(f)= \left\langle a_{1}\right\rangle_{L} + \left\langle a_{2}\right\rangle_{L} +\frac{1}{p^{L}}.
$$

\noindent \underline{Case 2:  $L=\infty$:}
Since $f^{\nu_{f}^{I_1+I_2}(p^{e})}\not\in (I_1+I_2)^{[p^{e}]}$, we conclude  that there exists $k_{1},k_{2}\in\mathbb{N}$ such that $k_{1}+k_{2}=\nu_{f}^{I_1+I_2}(p^{e})$ and $g_1^{k_{1}}g_2^{k_{2}}\not\in(I_1+I_2)^{[p^{e}]}$ by Lemma \ref{LemmaPorwerAndTerms}. Since the polynomials $g_1,g_2$ are in different sets of variables, we have 
$$
g_1^{k_{1}}\not\in I_1^{[p^{e}]},\quad g_2^{k_{2}}\not\in I_2^{[p^{e}]}.
$$
It follows that $k_{1}\leq \nu_{g_1}^{I_1}(p^{e}) $ and $k_{2}\leq \nu_{g_2}^{I_2}(p^{e}) $. Since 
$$
\nu_{f}^{I_1+I_2}(p^{e})=k_{1}+k_{2} \leq \nu_{g_1}^{I_1}(p^{e})+\nu_{g_2}^{I_2}(p^{e}), 
$$
we have that 
$$
\frac{\nu_{f}^{I_1+I_2}(p^{e})}{p^{e}} \leq \frac{\nu_{g_1}^{I_1}(p^{e})}{p^{e}}+\frac{\nu_{g_2}^{I_2}(p^{e})}{p^{e}}.
$$
Taking the limit when $e\rightarrow \infty$, we have  that
\begin{align*}
{\bf c}^{I_1+I_2}(f)&=\displaystyle\lim_{e\rightarrow \infty}\frac{\nu_{f}^{I_1+I_2}(p^{e})}{p^{e}}\leq \displaystyle\lim_{e\rightarrow \infty}\left(\frac{\nu_{g_1}^{I_1}(p^{e})}{p^{e}}+\frac{\nu_{g_2}^{I_2}(p^{e})}{p^{e}}\right)\\
&=\displaystyle\lim_{e\rightarrow \infty}\frac{\nu_{g_1}^{I_1}(p^{e})}{p^{e}}+\displaystyle\lim_{e\rightarrow \infty}\frac{\nu_{g_2}^{I_2}(p^{e})}{p^{e}} =a_{1}+a_{2}.
\end{align*}

We now check that $a_{1}+a_{2}\leq {\bf c}^{I_1+I_2}(f)$. Since  $a_{1}$ and $a_{2}$ add without carrying in base $p$,  Remark \ref{addwithout} and  Lemma \ref{2.5} imply that
\begin{equation}\label{9}
{
p^{e}(\left\langle a_{1}\right\rangle_{e} +\left\langle a_{2}\right\rangle_{e})\choose
p^{e}\left\langle a_{1}\right\rangle_{e}} \not\equiv 0\mod p.
\end{equation}
Moreover,
$\nu_{g_1}^{I_1}(p^{e})=p^{e}\left\langle a_{1} \right\rangle_{e}$, and $\nu_{g_2}^{I_2}(p^{e}) = p^{e}\left\langle a_{2}\right\rangle_{e}$ by Proposition \ref{tra}. Then,
\begin{equation}\label{EqMTW}
g_1^{p^{e}\left\langle a_{1} \right\rangle_{e}}\not\in I_1^{[p^{e}]}, \quad g_2^{p^{e}\left\langle a_{2}\right\rangle_{e}}\not\in I_2^{[p^{e}]}. 
\end{equation}
Equations \ref{9}, \ref{EqMTW}, and   Lemma \ref{LemmaPorwerAndTerms} imply that
$$
f^{p^{e}\left(\left\langle a_{1}\right\rangle_{e}+ \left\langle a_{2}\right\rangle_{e} \right) }\not\in (I_1+I_2)^{[p^{e}]}.
$$
Thus,
$$
p^{e}(\left\langle a_{1}\right\rangle_{e} +\left\langle a_{2}\right\rangle_{e})\leq \nu_{f}^{I_1+I_2}(p^{e})\qquad \hbox{ for all } e\gg 1.
$$
Dividing by $p^{e}$, we have
$$
\left\langle a_{1}\right\rangle_{e} +\left\langle a_{2}\right\rangle_{e}\leq\frac{\nu_{f}^{I_1+I_2}(p^{e})}{p^{e}}\qquad \hbox{ for all } e \gg 1.
$$
Finally, taking the limit when $e\rightarrow \infty$, we obtain that  $a_{1}+a_{2} \leq {\bf c}^{I_1+I_2}(f)$, and  ${\bf c}^{I_1+I_2}(f)=a_{1}+a_{2}$.
\end{proof}

\begin{example}\label{exqo}
Let $a,b,n \in \mathbb{Z}_{>0}$, and $b \geq a$. Let $g_1=x^n \in R_1= \KK[x], g_2=y_1^ay_2^b \in R_2= \KK[y_1, y_2]$, and $f=g_1+g_2\in  R_1 \otimes_\KK R_2$.  	Let $I_1= \m_1=(x)$, $I_2 = \m_2=(y_1,y_2)$, and $\m = \m_1+\m_2=(x, y_1, y_2 )$. Since 
	$a_1={\bf c}^{\m_1}(g_1)=\frac{1}{n}$ and ${\bf c}^{\m_2}(g_2)=\frac{1}{b}$, we have that  	\begin{align*}
{\bf c}^{\mathfrak{m}}(f)&=\left\lbrace \begin{array}{ll}
\frac{1}{n}+\frac{1}{b},& \text{if}\quad L=\infty,\\
&\\
\left\langle \frac{1}{n}\right\rangle_{L} + \left\langle \frac{1}{b}\right\rangle_{L} +\frac{1}{p^{L}},& \text{if}\quad L<\infty.
\end{array}\right.
	\end{align*}
where $L=\sup\left\lbrace N\in \mathbb{N}\quad|\quad\left( \frac{1}{n}\right) ^{(e)}+\left( \frac{1}{b}\right) ^{(e)}\leq p-1\quad\text{for all}\quad 0\leq e\leq N \right\rbrace$.
It is worth mentioning that ${\bf c}^{\mathfrak{m}}(f)={\bf c}^{\mathfrak{m}}(x^{n}+y_2^{b})$.
\end{example}

\begin{remark}\label{RemNumberFTbound}
If  $g_1\in I_1$ and $g_2\in I_2$, then $f\in I_1 R+I_2R$. Then, if we run over all the test ideal $I_i=\tau(g^{\lambda_i}_i)$ for $\lambda_i\in (0,1)$, we obtain all the the ideals that contain $g_i$ and whose $F$-thresholds give the jumping numbers in $(0,1)$.
Using Theorem 	\ref{4.2}, we compute the $F$-thresholds with respect to $\tau(g^{\lambda_1}_1)R+\tau(g^{\lambda_2}_2)R$. This give $F$-jumping numbers of $f$, but not necessarily all.
\end{remark}

\section{Relation with log canonical threshold}\label{SecLCT}

The behavior of the $F$-pure threshold for different reductions mod $p$ and their connections with invariants over $\QQ$ have attracted much interest. We refer the to recent surveys on this topic \cite{BFS,TWSurvey}.  We are particularly interested in the  relation between the log canonical threshold and the $F$-threshold. As consequence of Theorem \ref{4.2}, we  give further evidence for  the conjecture relating the $F$-pure and log-canonical threshold.

Let  $f\in\mathbb{Q}[x_1,\ldots,x_n]$, and let $f_{p}$ denote  the polynomial over $\mathbb{F}_{p}$ obtained by reducing each coefficient of $f$ modulo $p$.

A log resolution of $f$ defined over $\QQ$ is a proper birational map $\pi_\QQ : Y_\QQ \rightarrow \mathbb{A}^n_\QQ$, with $Y_\QQ$ smooth, such that $f\mathcal{O}_{Y_\QQ}$ is a principal and defines a simple normal crossing divisor $D$. The log canonical threshold ${\bf lct_{0}}(f)$ of $f$ at the origin $0\in \mathbb{A}^n_\QQ$ is the minimum number $\lambda \in \QQ_{>0}$ such that the ideal $(\pi_\QQ)_\ast \mathcal{O}_{Y_\QQ}(K_{\pi_\QQ} - \lfloor \lambda \cdot D \rfloor)$ is contained in the maximal ideal $(x_1, x_2 \dots, x_n)_\QQ$. 

Musta\c{t}\v{a}, Tagaki, and Watanabe showed the following relation between  the F-pure threshold and the log canonical threshold.

\begin{theorem}[{\cite[Theorem 3.4]{MTW}}]\label{MTWthm} 
Let  $f\in\mathbb{Q}[x_1,\ldots,x_n]$, and let $f_{p}$ denote  the polynomial over $\mathbb{F}_{p}$ obtained by reducing each coefficient of $f$ modulo $p$. Then, 
$\lim_{p\rightarrow\infty}{\bf c}^{\mathfrak{m}}(f_{p})={\bf lct_{0}}(f)$.
\end{theorem}

 Furthermore, they have formulated the following conjecture. 

\begin{conjecture}[{\cite[Conjecture 3.6]{MTW}}]\label{MTWconj} 
Let  $f\in\mathbb{Q}[x_1,\ldots,x_n]$, and let $f_{p}$ denote  the polynomial over $\mathbb{F}_{p}$ obtained by reducing each coefficient of $f$ modulo $p$.
There exist infinitely many prime numbers $p$ such that 
${\bf lct_{0}}(f)={\bf c}^{\mathfrak{m}}(f_{p}).$
\end{conjecture}

We now use Theorem \ref{4.2}  to study Conjecture \ref{MTWconj}.

\begin{theorem}\label{TS+MTW} 
We consider Notation \ref{NotationFT} with  $\KK=\QQ$.   Let $f=g_1+g_2 \in R_1 \otimes_\QQ R_2$ be a Thom-Sebastiani type polynomial as in the previous section. Assume that $g_1\in R_1$ and $g_2 \in R_2$ satisfy Conjeture \ref{MTWconj}. Denote by $B_i$, with $i=1,2$  the infinite set of primes $p$ such that ${\bf lct_{0}}(g_i)={\bf c}^{\mathfrak{m}}((g_i)_{p}).$ 
Assume that $a_1+ a_2 < 1$ and  there exists an infinite set $B \subseteq B_1\cap B_2$ such that  the non terminating base $p$ extensions of $a_1$ and $a_2$ add without carrying for every $p \in B$. Then    $f=g_1+g_2$ satisfies the Conjeture \ref{MTWconj}. 
\end{theorem}
\begin{proof}
This is a consequence of the Thom-Sebastiani property for log canonical thresholds \cite[Proposition 8.21]{MR1492525} and Theorem \ref{4.2}. 
\end{proof}

We know provide a series of examples where the hypothesis of Theorem \ref{TS+MTW} are satisfied.

\begin{example}[Thom-Sebastiani type sum of monomials] If $g_1 \in R_1$ and $g_2 \in R_2$  monomials, then $B_1=B_2$ is the set of all primes $p$. If $a_1+a_2<1$, Theorem \ref{2.6} implies the existence of an infinite set $B$ of primes $p$ for which the non terminating base $p$ expansions of $a_1$ and $a_2$ are periodic, and therefore add without carrying \cite[Lemma 4.16 and Example 4.4]{7}.  Then the hypothesis of Theorem \ref{TS+MTW} are satisfied, and Conjecture \ref{MTWconj} holds for $f$. Particular cases of this situation are diagonal hypersurfaces, and Thom-Sebastiani type binomials (or multinomials, since the argument extends to any finite number of summands).

A case of the latteris  Example \ref{exqo}. Notice that if ${\rm gcd}(a,b,n) \ne 1$ the hypersurfaces of Example \ref{exqo} are irreducible quasi-ordinary hypersurfaces with only one characteristic exponent.  
We point out that  the log-canonical thresholds of irreducible quasi-ordinary hypersurfaces has already been computed \cite{MR2957197}.
\end{example}

\begin{example}[Suspensions or ramified cyclic coverings]  Let  $g_1 \in R_1$ such that the set $B$ contains and infinite number of primes $p$ such that $(p-1)a_1 \in \mathbb{N}$. An example of this is  $g_1= x_{1}^{b_{1}}x_{2}^{b_{2}}+x_{1}^{c_{1}}x_{2}^{c_{2}}\in R_1=\mathbb{K}[x_{1},x_{2}]$  a binomial hypersurface whose maximal splitting polytope ${\bf P}_{f}$  contains a unique maximal point $\eta=(\frac{r_{1}}{s_{1}},\frac{r_{2}}{s_{2}})\in\mathbb{Q}^{2}$, with $\frac{r_{1}+1}{s_{1}}+\frac{r_{2}+1}{s_{2}}\leq 1$ )  \cite[Theorem 4.1]{HerBinom}. Let $g_{2} = y_1^{d} \in \mathbb{K}[y_{1}]$. Then $a_{2}= \frac{1}{d} = {\bf lct}_0(g_{2})$ for all primes, and so, $B_{2}$ is the set of all primes. Finally, by Theorem 2.8, there is a infinite subset $B$ of $B_1=B_1 \cap B_{2}$, consisting of all primes $p$ such that $(p-1)a_1,  (p-1)\frac{1}{d} \in \mathbb{N}$. We have  that, if $a_1+a_2<1$, then the non terminating base $p$ expansions of $a_1$ and $a_{2}$ add without carrying  \cite[Lemma 4.16 and Example 4.4]{7}. We conclude that ${\bf c}^{\mathfrak{m}}(f)= a_1+ a_1={\bf lct_{0}}(f)$ for all $p\in B$. Hence, $f=g_1+g_{2}$ satisfies Conjecture \ref{MTWconj}.
We note  that it is possible to iterated the previous construction because $f$ satisfies the hypothesis required on $g_1$. 
\end{example}

The following examples combine  the  previous cases.

\begin{example}
	Let $R_1=\mathbb{K}[x,y]$, $R_2=\mathbb{K}[w,z]$, $R_3=\mathbb{K}[t,u,v]$, and $g_1= x^{4}+y^{4}  \in \m_{1}=\left(x, y\right)$, $g_2= z^{7}w^{2}+z^{5}w^{6} \in \m_{2}=\left(z, w\right)$, and $g_3= v^{2}u^{3}t^{8} \in \m_{3}=(t, u, v)$. Set $f=g_1+g_2+g_3 \in R_1 \otimes_\KK R_2 \otimes_\KK R_3$, and $\m= \left(x, y, z, w, v, u, t \right)$. We compute the F-pure threshold of  $f$, and check that $f$ satisfies Conjecture \ref{MTWconj}. 
	\begin{itemize}
		\item If $p\equiv 1\mod 16$, Theorem \ref{4.2} implies that  ${\bf c}^{\mathfrak{m}_1}(g_1)=\frac{8}{16}={\bf lct}_{0}(g_1)$.
		\item If $p\equiv 1\mod 32$, then  ${\bf c}^{\mathfrak{m}_{2}}(g_2)=\frac{3}{16}$   \cite[Theorem 4.1]{HerBinom}. Moreover, Theorem \ref{MTWthm} implies that  ${\bf c}^{\mathfrak{m}_{2}}(g_2) = {\bf lct}_{0}(g_2)=\frac{3}{16}$.
		\item  We have that		${\bf c}^{\mathfrak{m}_{3}}(g_3)=\frac{2}{16}={\bf lct}_{0}(v^{2}u^{3}t^{8})$  for any prime $p$ 
	\end{itemize}
The set of primes $p$ such that $ p\equiv1\mod 32$ is infinite by Theorem \ref{2.6.1}. Therefore, there are infinite primes    such that ${\bf c}^{\mathfrak{m}}(f)=\frac{13}{16}={\bf lct}_{0}(f)$.
\end{example}

\begin{example}[Exploiting other properties of the log canonical threshold and the $F$-pure threshold] Let  $g_{1}$ and $g_{2}$ be polynomials 
satisfying  the conditions in Theorem \ref{TS+MTW}. 

We consider first  $f=(g_{1}+g_{2})^{n}$.  Then, for any prime $p$ in the infinite set $B$, we have  
$$
{\bf c}^{\mathfrak{m}}(f)=\frac{{\bf c}^{\mathfrak{m}}(g_{1})+{\bf c}^{\mathfrak{m}}(g_{2})}{n}=\frac{{\bf lct}_{0}(g_{1})+{\bf lct}_{0}(g_{2})}{n}={\bf lct}_{0}(f).
$$
Thus,  Conjecture \ref{MTWconj} holds for $f$. 

We consider now $f=g_{1}g_{2}$. Then, we have
$${\bf c}^{\mathfrak{m}}(f)\leq{\bf lct}_{0}(f)\leq\min\left\lbrace {\bf lct}_{0}(g_{1}),{\bf lct}_{0}(g_{2})\right\rbrace=\min\left\lbrace {\bf lct}_{0}(g_{1}),{\bf lct}_{0}(g_{2})\right\rbrace = {\bf c}^{\mathfrak{m}}(f)$$ 
 for any integer $p$ in the infinite set $B$  \cite[Proposition 3.2]{TW} \cite[Theorem 8.20]{MR1492525} \cite[Lemma 3.3]{HerBinom}.

Hence,  ${\bf c}^{\mathfrak{m}}(f)={\bf lct}_{0}(f)$ for any prime $p$ in $B$, and  so, Conjecture \ref{MTWconj} holds for $f$. 
\end{example}

\begin{example}
Consider the polynomial
$$
f=w^{8}x^{4}y^{24}z^{44}+4u^{3}v^{7}w^{8}x^{4}y^{24}z^{44}+6u^{6}v^{14}w^{4}x^{2}y^{12}z^{22}+4u^{9}v^{21}w^{2}x^{1}y^{6}z^{11}+u^{12}v^{28},
$$ 
and $p\equiv 2\mod 77$. We note that  that $f=\left( u^{3}v^{7}+w^{2}xy^{6}z^{11}\right)^{4}=(g_{1}+g_{2})^{4}$
for $g_{1}=u^{3}v^{7}\in\KK[u,v]$ and $g_{2}=w^{2}xy^{6}z^{11}\in\KK[w,x,y,z]$. Then,   ${\bf c}^{\mathfrak{m}}(f)={\bf c}^{\mathfrak{m}}(g_1+g_2)/4$, where $\mathfrak{m}= \left(u, v, w, x, y, z\right)$.  For the monomials  $g_{1}$, $g_{2}$ we get that 
$${\bf c}^{\mathfrak{m}_{1}}(g_{1})=1/7={\bf lct}_{0}(g_{1}), \, \hbox{ and } \,
{\bf c}^{\mathfrak{m}_{2}}(g_{2})=1/11={\bf lct}_{0}(g_{2}),$$
where $\mathfrak{m}_{1}=\left(u, v\right)$, and $\mathfrak{m}_{2}=\left(w, x, y, z\right)$. The number $1/7$ and $1/11$ add without carrying whenever $p\equiv2\mod77$, then  ${\bf c}^{\mathfrak{m}}(f)=1/4(1/7+1/11)=9/154$ by Theorem \ref{4.2}. There are infinitely many prime numbers that satisfies  $p\equiv2\mod77$, and 
$$
{\bf lct}_{0}(f)={\bf lct}_{0}(g_{1}+g_{2})/4=9/154={\bf c}^{\mathfrak{m}}(f).
$$
Therefore Conjecture \ref{MTWconj} holds for the polynomial $f$.
\end{example}

\begin{example}
Let $f=w^{4}x^{6}t^{3}+uv^{2}w^{4}x^{6}+t^{3}y^{2}z+y^{2}zuv^{2}\in\KK[t,u,v,w,x,y,z]$, and $p\equiv1\mod6$. Notice that 
$$
f=(w^{4}x^{6}+y^{2}z)(t^{3}+uv^{2})=g_{1}g_{2},
$$
for $g_{1}=w^{4}x^{6}+y^{2}z\in\KK[w,x,y,z]$, and $g_{2}=t^{3}+uv^{2}\in\KK[t,u,v]$.  We have that ${\bf c}^{\mathfrak{m}}(f)=\min\left\lbrace {\bf c}^{\mathfrak{m}_{1}}(g_{1}),{\bf c}^{\mathfrak{m}_{2}}(g_{2})\right\rbrace $, where $\mathfrak{m}_{1}=\left(w, x, y, z\right)$ and $\mathfrak{m}_{2}=\left(t, u, v\right)$. Using Theorem \ref{4.2}, we compute the $F$-threshold of $g_{1}$, and $g_{2}$. For $p\equiv1\mod6$ we have that $1/6$, and $1/2$ add without carrying. Then, ${\bf c}^{\mathfrak{m}_{1}}(g_{1})=1/6+1/2=2/3$. By Example \ref{exqo}, we obtain ${\bf c}^{\mathfrak{m}_{2}}(g_{2})=5/6$. Therefore, ${\bf c}^{\mathfrak{m}}(f)={\bf c}^{\mathfrak{m}_{1}}(g_{1})=2/3$.  There are infinitely many prime numbers that satisfies  $p\equiv1\mod6$, and  
$$
{\bf c}^{\mathfrak{m}}(f)\leq{\bf lct}_{0}(f)\leq\min\left\lbrace {\bf lct}_{0}(g_{1}),{\bf lct}_{0}(g_{2})\right\rbrace =2/3.
$$
Therefore Conjecture \ref{MTWconj} holds for the polynomial $f$.
\end{example}

\section{Test ideal of a Thom-Sebastiani type polynomial}\label{5}

In this section we compute a formula for the first non-trivial test ideal of a  Thom-Sebastiani type polynomial.

\begin{notation}\label{NotationTest}
 Let $\{x_1, \dots, x_n\}$ and $\{ y_1, \dots, y_m\}$ be two disjoint sets of variables. Consider the rings of polynomials $R_{1}=\KK[x_{1},\dots,x_{n}]$ and $R_{2}=\KK[y_{1},\dots,y_{m}]$. Denote by $\mathfrak{m}_{1}=\left( x_{1},\dots,x_{n}\right)$ and $\mathfrak{m}_{2}=\left( y_{1},\dots,y_{m}\right)$ their maximal homogeneous ideals.
  Let $f=g_1+g_2 \in R=R_1 \otimes_\KK R_2=\KK[x_1, \dots, x_n, y_1, \dots, y_m]$, where  $g_1\in R_{1}$ and $g_2\in R_{2}$. Let $a_{1}={\bf c}^{\mathfrak{m}_{1}}(g_{1})$, and $a_{2}={\bf c}^{\mathfrak{m}_{2}}(g_{2})$.  
\end{notation}

We start with some of preparation lemmas before proving Theorem \ref{B}.

\begin{lemma}\label{lemma1}
Let $\alpha_{1},\alpha_{2}\in\mathbb{R}_{\geq0}$. Assume that $\alpha_{1}^{(e)}+\alpha_{2}^{(e)}\leq p-2$, and 
$$
{p^{e}(\left\langle \alpha_{1}\right\rangle_{e} +\left\langle \alpha_{2}\right\rangle_{e}) \choose
p^{e} \left\langle \alpha_{1}\right\rangle_{e}}
\not\equiv 0\mod p,
$$
then 
$$
{p^{e}(\left\langle \alpha_{1}\right\rangle_{e} +\left\langle \alpha_{2}\right\rangle_{e}) +1\choose
p^{e}\left\langle \alpha_{1}\right\rangle_{e}+1}\not\equiv0\mod p
\quad \& \quad
{p^{e}(\left\langle \alpha_{1}\right\rangle_{e} +\left\langle \alpha_{2}\right\rangle_{e}) +1\choose
p^{e}\left\langle \alpha_{2}\right\rangle_{e}+1}\not\equiv0\mod p.
$$

\end{lemma}

\begin{proof} From the definition of binomial coefficient we get  the equalities
$$
{p^{e}(\left\langle \alpha_{1}\right\rangle_{e} +\left\langle \alpha_{2}\right\rangle_{e}) +1\choose
p^{e} \left\langle \alpha_{i}\right\rangle_{e}+1}={p^{e}(\left\langle \alpha_{1}\right\rangle_{e} +\left\langle \alpha_{2}\right\rangle_{e}) \choose
p^{e}\left\langle \alpha_{i}\right\rangle_{e}}\cdot\frac{p^{e}(\left\langle \alpha_{1}\right\rangle_{e} +\left\langle \alpha_{2}\right\rangle_{e})+1}{p^{e}\langle \alpha_{i}\rangle_{e} +1},
$$
with $i=1,2$.
The result follows from these equalities and  the congruence 
$$p^{e}(\left\langle \alpha_{1}\right\rangle_{e} +\left\langle \alpha_{2}\right\rangle_{e})+1\equiv \alpha_{1}^{(e)}+\alpha_{2}^{(e)}+1 \not\equiv 0\mod p.$$ 
The latter holds because $\alpha_{i}^{(e)}\leq \alpha_{1}^{(e)}+\alpha_{2}^{(e)}\leq p-2$. 
\end{proof}

\begin{lemma}\label{lemma-d}
Let $\alpha_{1},\alpha_{2}\in\mathbb{R}_{\geq0}$, 
and assume that $\alpha_{1}+\alpha_{2}<1$.
Let
$$
L=\sup\left\lbrace N\in \mathbb{N}\quad|\quad \alpha_{1}^{(e)}+\alpha_{2}^{(e)}\leq p-1\quad\text{for all}\quad 0\leq e\leq N \right\rbrace
$$
and 
$$d=\sup\left\lbrace e\leq L\quad|\quad \alpha_{1}^{(e)}+\alpha_{2}^{(e)}\leq p-2\right\rbrace.$$ 
Then,
 $$
 \left\langle \alpha_{1}\right\rangle_{d} + \left\langle \alpha_{2}\right\rangle_{d} +\frac{1}{p^{d}}=\left\langle \alpha_{1}\right\rangle_{L} + \left\langle\alpha_{2}\right\rangle_{L} +\frac{1}{p^{L}}.
 $$
\end{lemma}
\begin{proof}
We proceed by cases.

If $d=L=\infty$, then the result follows because $\langle \alpha_i\rangle_{\infty}=\alpha_i$.

If $L=\infty$ and $d<\infty$, we have $\alpha_1^{(e)}+\alpha_2^{(e)}=p-1$ for $e\geq d+1 $. Hence
$$
\alpha_1+\alpha_2=\left\langle \alpha_{1}\right\rangle_{L} + \left\langle \alpha_{2}\right\rangle_{L} +\frac{1}{p^{L}}
=\left\langle \alpha_{1}\right\rangle_{d} + \left\langle \alpha_{2}\right\rangle_{d} +\sum^\infty_{j=d+1}\frac{p-1}{p^{j}}
=\left\langle \alpha_{1}\right\rangle_{d} + \left\langle \alpha_{2}\right\rangle_{d} +\frac{1}{p^{d}}.
$$

If $L<\infty$ and $d<\infty$, we have $\alpha_1^{(e)}+\alpha_2^{(e)}=p-1$ for $d+1 \leq e \leq L$. Hence
$$
\left\langle \alpha_{1}\right\rangle_{L} + \left\langle \alpha_{2}\right\rangle_{L} +\frac{1}{p^{L}}
=\left\langle \alpha_{1}\right\rangle_{d} + \left\langle \alpha_{2}\right\rangle_{d} +\frac{p-1}{p^{d+1}}+\dots+\frac{p-1}{p^{L}}+\frac{1}{p^{L}}
=\left\langle \alpha_{1}\right\rangle_{d} + \left\langle \alpha_{2}\right\rangle_{d} +\frac{1}{p^{d}}.$$
\end{proof}

\begin{lemma}\label{Lemma-p-2}
Consider Notation \ref{NotationTest}, and assume that $a_{1}+a_{2}<1$.  Let
$$L=\sup\left\lbrace N\in \mathbb{N}\quad|\quad a_{1}^{(e)}+a_{2}^{(e)}\leq p-1\quad\text{for all}\quad 0\leq e\leq N \right\rbrace.$$
If $a_{1}^{(e)}+a_{2}^{(e)}\leq p-2$ for some $e\leq L$, then 
$$
\left( f^{ p^{e}(\langle a_{1}\rangle_{e}+\langle a_{2}\rangle_{e})+1}\right)^{[1/p^{e}]}=\left( g_{1}^{\lceil p^{e}a_{1}\rceil}\right)^{[1/p^{e}]}+\left( g_{2}^{\lceil p^{e}a_{2}\rceil}\right)^{[1/p^{e}]}.
$$
\end{lemma}
\begin{proof}
Let $\theta_i= p^{e}\langle a_i\rangle_{e}= \nu_{g_{i}}^{\m_i}(p^e)$ for $i=1,2$,   and 
$\theta= p^{e}(\langle a_1\rangle_{e}+ \langle a_2\rangle_{e})=\theta_1+\theta_2.$ 

First, we check that
$$
\left( g_{1}^{\lceil p^{e}a_{1}\rceil}\right)^{[1/p^{e}]}+\left( g_{2}^{\lceil p^{e}a_{2}\rceil}\right)^{[1/p^{e}]}\subseteq \left( f^{ p^{e}(\langle a_{1}\rangle_{e}+\langle a_{2}\rangle_{e})+1}\right)^{[1/p^{e}]}
$$ 
for $i=1,2$.
We choose elements $w_{i,j}\in R_2$ such that $1,g_i,\ldots, g^{\theta_i}_i, w_{i,1},\dots,w_{i,s_i}$ is  a free basis of $R_{i}$ as $R_{i}^{p^{e}}$-module.
Then,  
$$
\mathcal{B}_e  =\left\{  g^{j_1}_1g^{j_2}_2,g^{j_1}_1 w_{2,k_2}, w_{1,k_1}g^{j_2}_2, w_{1,k_1}w_{2,k_2} \; |\;
0\leq j_i\leq \theta_i \; \&\; 0\leq k_i\leq s_i  \right\}\\
$$
is a free basis of  $R$ over $R^{p^{e}}$.
We fix $h_{j}\in R_{1}$ such that
$$
 g_{1}^{\theta_1+1}=\sum^{\theta_1}_{j=0}h_{j}^{p^{e}}g^j_{1}+\sum^{s_1}_{j=1}h_{\theta+j}^{p^{e}}w_{1,j}.
$$
Then, 
$$\left( g_{1}^{\lceil p^{e}a_{1}\rceil}\right)^{[1/p^{e}]}=(g_1^{p^e\langle a_1\rangle_e+1 })^{[1/p^e]} =\langle h_{j}\;| \;0\leq j\leq \theta_1+s_1 \rangle$$ 
by  Proposition \ref{compute Frobeniusroot}.
We  consider  $f^{\theta+1}$ in terms of the basis $\mathcal{B}_e$. We now show that each $h_{j}^{p^{e}}$ appears as  a coefficient.
Then,
\begin{align*}
f^{\theta+1}&=\sum^{\theta+1}_{k=0} {\theta+1 \choose k} g_{1}^{k}g_{2}^{\theta+1-k}
=
g_{2}^{\theta+1}+\dots
+{\theta+1\choose \theta_1+1} g_1^{\theta_1+1} g_{2}^{\theta_2}
+\dots+g_{1}^{\theta+1}.
\end{align*}
By Lemma \ref{lemma1}, we have 
$${\theta+1 \choose \theta_i+1} \not\equiv 0\mod p.$$
We  expand  $g^{\theta_1+1}_1 g_{2}^{\theta_2}$ as 
$$
g_{1}^{\theta_1+1} g_{2}^{\theta_2}
=     (\sum^{\theta_1}_{j=0}h_{j}^{p^{e}}g^j_{1}+\sum^{s_1}_{j=1}h_{\theta+j}^{p^{e}}w_{1,j})       g_{2}^{\theta_2}
=\sum^{\theta_1}_{j=0}h_{j}^{p^{e}}g^j_{1}g_{2}^{\theta_2}+\sum^{s_1}_{j=1}h_{\theta+j}^{p^{e}}w_{1,j}g_{2}^{\theta_2}.
$$
We note that the other summands are either multiples of  $w_{1,k_1}g_{2}^{j_2}$, or  $g_{1}^{j_1}w_{2,k_2}$, 
with 
$0\leq j_1\leq \theta_1$, $0\leq j_2<\theta_2$, and $0\leq k_i\leq s_i$. Then,    $h_{j}\in \left( f^{\theta+1}\right)^{[1/p^{e}]} $.  We conclude that $\left( g_{1}^{\lceil p^{e}a_{1}\rceil}\right)^{[1/p^{e}]}\subseteq \left( f^{ p^{e}(\langle a_{1}\rangle_{e}+\langle a_{2}\rangle_{e})+1}\right)^{[1/p^{e}]}$.
Similarly, $\left( g_{2}^{\lceil p^{e}a_{2}\rceil}\right)^{[1/p^{e}]}\subseteq \left( f^{ p^{e}(\langle a_{1}\rangle_{e}+\langle a_{2}\rangle_{e})+1}\right)^{[1/p^{e}]}$.

We now show the other containment. 
From the expression 
\begin{align*}
f^{\theta+1}
=
g_{2}^{\theta+1}+{\theta+1\choose 1}g_{1}g_{2}^{\theta}+\dots
+{\theta+1\choose \theta_1} g_1^{\theta_1} g_{2}^{\theta_2+1}+{\theta+1\choose \theta_1+1} g_1^{\theta_1+1} g_{2}^{\theta_2}
+\dots+g_{1}^{\theta+1},
\end{align*}
we conclude  that 
$
f^{\theta+1}\in \left( g_{1}^{\theta_{1}+1},  g_{2}^{\theta_{2}+1}\right).
$
Then,
$$
\left( f^{\theta+1}\right)^{[1/p^{e}]} \subseteq \left(   g_{1}^{\theta_1+1}, g_{2}^{\theta_2+1}\right) ^{[1/p^{e}]}
=  \left( g_{1}^{\lceil p^{e} a_{1}\rceil}\right)^{[1/p^{e}]}+\left( g_{2}^{\lceil p^{e} a_{2}\rceil}\right)^{[1/p^{e}]}.
$$
\end{proof}
We are now ready to  prove Theorem \ref{B}. 
\begin{theorem}\label{456}
Consider Notation \ref{NotationTest}.
Let
$$
L=\sup\left\lbrace N\in \mathbb{N}\quad|\quad a_{1}^{(e)}+a_{2}^{(e)}\leq p-1\quad\text{for all}\quad 0\leq e\leq N \right\rbrace
$$
and 
$d=\sup\left\lbrace e\leq L\quad|\quad a_{1}^{(e)}+a_{2}^{(e)}\leq p-2\right\rbrace.$ 
Then 
$$
\tau\left( f^{{\bf c}^{\mathfrak{m}}(f)}\right) =\left\lbrace \begin{array}{cl}
(f)&\text{if } {\bf c}^{\mathfrak{m}}(f)=1,  \\
\tau(g^{a_1}_{1})+\tau(g^{a_2}_{2})&\text{if } {\bf c}^{\mathfrak{m}}(f)\not\in p^{-e}\cdot\mathbb{N},\\
\left( g_{1}^{\lceil p^{d} a_{1}\rceil}\right)^{[1/p^{d}]}+\left( g_{2}^{\lceil p^{d} a_{2}\rceil}\right)^{[1/p^{d}]}&\text{if }  {\bf c}^{\mathfrak{m}}(f)\in \ZZ[\frac{1}{p}]\; \& \;{\bf c}^{\mathfrak{m}}(f)\neq 1 .
\end{array}\right.
$$
\end{theorem}
\begin{proof} 
It is well know that $\tau(f^1)=(f)$. This settles the first case. We assume that ${\bf c}^{\mathfrak{m}}(f)\not=1$. 

If ${\bf c}^{\mathfrak{m}}(f)\not\in \ZZ[\frac{1}{p}]$, then $a_1$ and $a_2$ add without carrying by Theorem \ref{A}. Thus, $L=\infty$, and ${\bf c}^{\mathfrak{m}}(f)=a_1+a_2$.
There exists infinity $e\in\mathbb{N}$ such that $a_{1}^{(e)}+a_{2}^{(e)}\leq  p-2$, 
because  ${\bf c}^{\mathfrak{m}}(f)\not\in \ZZ[\frac{1}{p}]$.
The result follows from Lemma \ref{Lemma-p-2}, by picking $e\in\NN$ such that $a_{1}^{(e)}+a_{2}^{(e)}\leq  p-2$,
$\tau(g^{a_1}_1)=\left(g^{\lceil p^e a_1 \rceil}_1\right)^{[1/p^e]}$,
$\tau(g^{a_2}_2)=\left(g^{\lceil p^e a_2 \rceil}_2\right)^{[1/p^e]}$,
and 
$$
\tau(f^{a_1+a_2})
=
\left(f^{\lceil p^e( a_1+a_2) \rceil}\right)^{[1/p^e]}=\left(f^{p^e\left( \langle a_1\rangle_e+\langle a_2\rangle_e\right) +1}\right)^{[1/p^e]}.
$$

Finally, suppose that  ${\bf c}^{\mathfrak{m}}(f)\in p^{-e}\cdot\mathbb{N}$. Lemma \ref{lemma-d} implies that 
$${\bf c}^{\mathfrak{m}}(f)=\left\langle a_{1}\right\rangle_{L} + \left\langle a_{2}\right\rangle_{L} +\frac{1}{p^{L}}=\left\langle a_{1}\right\rangle_{d} + \left\langle a_{2}\right\rangle_{d} +\frac{1}{p^{d}}.$$ Then, 
\begin{align*}
\tau\left( f^{{\bf c}^{\mathfrak{m}}(f)}\right)=\tau\left(f^{\left\langle a_{1}\right\rangle_{d} + \left\langle a_{2}\right\rangle_{d} +\frac{1}{p^{d}}} \right)&=\left( f^{ p^{d}(\langle a_{1}\rangle_{d}+\langle a_{2}\rangle_{d})+1}\right)^{[1/p^{d}]}=\left( g_{1}^{\lceil p^{d}a_{1}\rceil}\right)^{[1/p^{d}]}+\left( g_{2}^{\lceil p^{d}a_{2}\rceil}\right)^{[1/p^{d}]}
\end{align*}
by  Proposition \ref{PropTestIdealPpower} and Lemma \ref{Lemma-p-2}.
\end{proof}

\begin{example}[{\rm \cite[Example 4.9]{HerDiag}}]\label{ExNotHdz}
Let $p>2$ a prime number,  $a=\frac{p-1}{2}$, and $b=\frac{p^2-1}{2}$. 
Consider the polynomials $g_{1}=x_{1}^{b}+\dots+x_{a}^{b} \in R_1=\mathbb{F}_{p}[x_{1},\dots,x_{a}]$, $g_{2}=y_{1}^{pb}+\dots+y_{a}^{pb} \in R_2= \mathbb{F}_{p}[y_{1},\dots,y_{a}]$, and $f=g_{1}+g_{2}\in R= R_1 \otimes_{\mathbb{F}_{p}} R_2.$
Let $\mathfrak{m}_{1}=\langle x_{1},\dots,x_{a}\rangle \subset R_1$,  $\mathfrak{m}_{2}=\langle y_{1},\dots,y_{a}\rangle \subset R_2$ and $\mathfrak{m}=\langle x_{1},\dots,x_{a},y_{1},\dots,y_{a}\rangle \subset R$   be the respective maximal homogeneous ideals. Theorem \ref{4.2} implies that ${\bf c}^{\mathfrak{m}}(f)=a_{1}+a_{2}=\frac{1}{p}\in p^{-1}\cdot\mathbb{N}$, where $a_{1}={\bf c}^{\mathfrak{m}_{1}}(g_{1})=a/b$, and $a_{2}={\bf c}^{\mathfrak{m}_{2}}(g_{2})=a/pb$.
Since
$$
g_{1}=x_{1}^{b}+\dots+x_{a}^{b}=(x_{1}^{a})^{p}\cdot x_1^{a}+\dots+(x_{a}^{a})^{p}\cdot x_a^{a},\quad \hbox{ and } \quad  g_{2}=(y_{1}^{b}+\dots+y_{a}^{b})^{p}\cdot1,
$$
we conclude from Theorem \ref{456}, that 
$$
\tau\left( f^{{\bf c}^{\mathfrak{m}}(f)}\right)=\left( g_{1}^{\lceil pa_{1}\rceil}\right)^{[1/p]}+\left( g_{2}^{\lceil pa_{2}\rceil}\right)^{[1/p]}=\langle x_{1}^{a},\dots,x_{a}^{a}\rangle+\langle y_{1}^{b}+\dots+y_{a}^{b}\rangle=\langle x_{1}^{a},\dots,x_{a}^{a},y_{1}^{b}+\dots+y_{a}^{b}\rangle.
$$
\end{example}

\begin{example}\label{extestide}
Consider the following polynomial $f=z^7w^2+z^5w^6+v^2u^3t^8\in \KK[t,u,v,w,z]$. For $p=97$, we want to compute the test ideal of $f$ at $a={\bf c}^\mathfrak{m}(f)$, where $\mathfrak{m}=\left(t, u, v, w, z\right)$. We split the polynomial $f$ into $g_{1}=z^7w^2+z^5w^6\in \KK[w,z]$, and $g_{2}=v^2u^3t^8\in \KK[t,u,v]$. Let $a_{1}, a_{2}$ be the $F$-threshold of $g_{1}$, and $g_{2}$, respectively.  We know that $a_{1}=3/16$, and  $a_{2}=1/8$, because $g_{2}$ is a monomial \cite[Example 4.3]{HerBinom}. Since, $a_{1}$, and $a_{2}$ add without carrying, by Theorem \ref{4.2} we have that $a=a_{1}+a_{2}=5/16$. The test ideals $\tau(g_{1}^{a_{1}})=\left(w, z\right)$, and $\tau(g_{2}^{a_{2}})=\left(t\right)$ were computed using Definition \ref{testid}. 
Finally, Theorem \ref{456} implies that $\tau\left( f^{a}\right)=\tau(g_{1}^{a_{1}})+\tau(g_{2}^{a_{2}})=\left(t, w, z\right)$.
\end{example}

\section*{Acknowledgments}
The authors thank Daniel J. Hernández for helpful comments and suggestions.
The second author started this  work while pursuing a master's degree at CIMAT. He thanks this institution for its support during his studies. We use Macaulay2 \cite{17} to compute several examples.

\bibliographystyle{alpha}
\bibliography{References}

\begin{thebibliography}{HNnBWZ16}

\bibitem[BFS13]{BFS}
Ang\'{e}lica Benito, Eleonore Faber, and Karen~E. Smith.
\newblock Measuring singularities with {F}robenius: the basics.
\newblock In {\em Commutative algebra}, pages 57--97. Springer, New York, 2013.

\bibitem[BGPGV12]{MR2957197}
Nero Budur, Pedro~D. Gonz\'{a}lez-P\'{e}rez, and Manuel Gonz\'{a}lez~Villa.
\newblock Log canonical thresholds of quasi-ordinary hypersurface
  singularities.
\newblock {\em Proc. Amer. Math. Soc.}, 140(12):4075--4083, 2012.

\bibitem[BMS08]{BMS}
Manuel Blickle, Mircea Musta\c{t}\v{a}, and Karen~E. Smith.
\newblock Discreteness and rationality of {$F$}-thresholds.
\newblock {\em Michigan Math. J.}, 57:43--61, 2008.
\newblock Special volume in honor of Melvin Hochster.

\bibitem[BMS09]{BMS-TAMS}
Manuel Blickle, Mircea Musta\c{t}\u{a}, and Karen~E. Smith.
\newblock {$F$}-thresholds of hypersurfaces.
\newblock {\em Trans. Amer. Math. Soc.}, 361(12):6549--6565, 2009.

\bibitem[BS15]{16}
Bhargav Bhatt and Anurag~K. Singh.
\newblock The {$F$}-pure threshold of a {C}alabi-{Y}au hypersurface.
\newblock {\em Math. Ann.}, 362(1-2):551--567, 2015.

\bibitem[DL99]{DL}
Jan Denef and Fran\c{c}ois Loeser.
\newblock Motivic exponential integrals and a motivic {T}hom-{S}ebastiani
  theorem.
\newblock {\em Duke Math. J.}, 99(2):285--309, 1999.

\bibitem[DSNnBP18]{DSNBP}
Alessandro De~Stefani, Luis N\'{u}\~{n}ez Betancourt, and Felipe P\'{e}rez.
\newblock On the existence of {$F$}-thresholds and related limits.
\newblock {\em Trans. Amer. Math. Soc.}, 370(9):6629--6650, 2018.

\bibitem[GS]{17}
Daniel~R. Grayson and Michael~E. Stillman.
\newblock Macaulay2, a software system for research in algebraic geometry.
\newblock Available at \url{http://www.math.uiuc.edu/Macaulay2/}.

\bibitem[Her11]{7}
Daniel~Jesus Hernandez.
\newblock {\em F-purity of hypersurfaces}.
\newblock ProQuest LLC, Ann Arbor, MI, 2011.
\newblock Thesis (Ph.D.)--University of Michigan.

\bibitem[Her14]{HerBinom}
Daniel~J. Hern\'{a}ndez.
\newblock {$F$}-pure thresholds of binomial hypersurfaces.
\newblock {\em Proc. Amer. Math. Soc.}, 142(7):2227--2242, 2014.

\bibitem[Her15]{HerDiag}
Daniel~J. Hern\'{a}ndez.
\newblock {$F$}-invariants of diagonal hypersurfaces.
\newblock {\em Proc. Amer. Math. Soc.}, 143(1):87--104, 2015.

\bibitem[HMTW08]{HMTW}
Craig Huneke, Mircea Musta\c{t}\u{a}, Shunsuke Takagi, and Kei-ichi Watanabe.
\newblock F-thresholds, tight closure, integral closure, and multiplicity
  bounds.
\newblock volume~57, pages 463--483. 2008.
\newblock Special volume in honor of Melvin Hochster.

\bibitem[HNnBWZ16]{HNBWZ}
Daniel~J. Hern\'{a}ndez, Luis N\'{u}\~{n}ez Betancourt, Emily~E. Witt, and
  Wenliang Zhang.
\newblock {$F$}-pure thresholds of homogeneous polynomials.
\newblock {\em Michigan Math. J.}, 65(1):57--87, 2016.

\bibitem[HY03]{HY}
Nobuo Hara and Ken-Ichi Yoshida.
\newblock A generalization of tight closure and multiplier ideals.
\newblock {\em Trans. Amer. Math. Soc.}, 355(8):3143--3174, 2003.

\bibitem[Kol97a]{Kollar}
J\'{a}nos Koll\'{a}r.
\newblock Singularities of pairs.
\newblock In {\em Algebraic geometry---{S}anta {C}ruz 1995}, volume~62 of {\em
  Proc. Sympos. Pure Math.}, pages 221--287. Amer. Math. Soc., Providence, RI,
  1997.

\bibitem[Kol97b]{MR1492525}
J\'{a}nos Koll\'{a}r.
\newblock Singularities of pairs.
\newblock In {\em Algebraic geometry---{S}anta {C}ruz 1995}, volume~62 of {\em
  Proc. Sympos. Pure Math.}, pages 221--287. Amer. Math. Soc., Providence, RI,
  1997.

\bibitem[L.E02]{Dic}
L.E.Dickson.
\newblock Theorems on the residues of multinomial coefficients with respect to
  a prime modulus.
\newblock {\em Quarterly Jornal of Pure and Applied Mathematics}, 1902.

\bibitem[Luc78]{Luc78}
Edouard Lucas.
\newblock Theorie des fonctions numeriques simplement periodiques.
\newblock {\em Amer. J. Math.}, 1878.

\bibitem[MSS20]{MSS}
Laurentiu Maxim, Morihiko Saito, and J\"{o}rg Sch\"{u}rmann.
\newblock {T}hom-{S}ebastiani theorems for filtered {$\mathcal{D}$}-modules and
  for multiplier ideals.
\newblock {\em Int. Math. Res. Not. IMRN}, (1):91--111, 2020.

\bibitem[MTW05]{MTW}
Mircea Musta\c{t}\v{a}, Shunsuke Takagi, and Kei-ichi Watanabe.
\newblock F-thresholds and {B}ernstein-{S}ato polynomials.
\newblock In {\em European {C}ongress of {M}athematics}, pages 341--364. Eur.
  Math. Soc., Z\"{u}rich, 2005.

\bibitem[Pag18]{Pagi}
Gilad Pagi.
\newblock An elementary computation of the {$F$}-pure threshold of an elliptic
  curve.
\newblock {\em J. Algebra}, 515:328--343, 2018.

\bibitem[SS85]{SS}
J.~Scherk and J.~H.~M. Steenbrink.
\newblock On the mixed {H}odge structure on the cohomology of the {M}ilnor
  fibre.
\newblock {\em Math. Ann.}, 271(4):641--665, 1985.

\bibitem[ST71]{ST}
M.~Sebastiani and R.~Thom.
\newblock Un r\'{e}sultat sur la monodromie.
\newblock {\em Invent. Math.}, 13:90--96, 1971.

\bibitem[TW04]{TW}
Shunsuke Takagi and Kei-ichi Watanabe.
\newblock On {F}-pure thresholds.
\newblock {\em J. Algebra}, 282(1):278--297, 2004.

\bibitem[TW18]{TWSurvey}
Shunsuke Takagi and Kei-Ichi Watanabe.
\newblock {$F$}-singularities: applications of characteristic {$p$} methods to
  singularity theory [translation of {MR}3135334].
\newblock {\em Sugaku Expositions}, 31(1):1--42, 2018.

\end{thebibliography}

\end{document}